\documentclass[12pt]{amsart}

\usepackage{amssymb}
\usepackage{amsmath}
\usepackage[colorlinks=true]{hyperref}
\usepackage[all]{xy}
\usepackage{aliascnt}

\newcommand{\B}{\mathrm{B}}
\newcommand{\C}{\mathrm{C}}
\newcommand{\bC}{\mathbf{C}}
\newcommand{\bL}{\mathbb{L}}
\newcommand{\N}{\mathrm{N}}
\newcommand{\bT}{\mathbb{T}}
\newcommand{\T}{\mathrm{T}}

\newcommand{\cC}{\mathcal{C}}
\newcommand{\cN}{\mathcal{N}}
\newcommand{\cO}{\mathcal{O}}

\renewcommand{\b}{\mathfrak{b}}
\newcommand{\g}{\mathfrak{g}}
\newcommand{\h}{\mathfrak{h}}
\newcommand{\m}{\mathfrak{m}}
\newcommand{\n}{\mathfrak{n}}
\newcommand{\p}{\mathfrak{p}}
\renewcommand{\u}{\mathfrak{u}}

\renewcommand{\d}{\mathrm{d}}
\newcommand{\ddr}{\mathrm{d}_{\mathrm{dR}}}

\newcommand{\Ad}{\mathrm{Ad}}
\newcommand{\aff}{\mathrm{aff}}
\newcommand{\Bun}{\mathrm{Bun}}

\newcommand{\coev}{\mathrm{coev}}
\newcommand{\DR}{\mathrm{DR}}
\newcommand{\ev}{\mathrm{ev}}

\newcommand{\id}{\mathrm{id}}
\newcommand{\IsotCorr}{\mathrm{IsotCorr}}
\newcommand{\LagrCorr}{\mathrm{LagrCorr}}
\newcommand{\Lie}{\mathrm{Lie}}
\newcommand{\Map}{\mathrm{Map}}
\newcommand{\pt}{\mathrm{pt}}
\newcommand{\Sym}{\mathrm{Sym}}
\newcommand{\QCoh}{\mathrm{QCoh}}

\DeclareMathOperator{\Rep}{Rep}
\DeclareMathOperator{\Spec}{Spec}

\newtheorem{thm}{Theorem}[section]
\newtheorem*{theorm}{Theorem}
\newaliascnt{lm}{thm}
\newaliascnt{prop}{thm}
\newaliascnt{cor}{thm}
\newtheorem{lm}[lm]{Lemma}
\newtheorem{prop}[prop]{Proposition}
\newtheorem{cor}[cor]{Corollary}
\aliascntresetthe{lm}
\aliascntresetthe{prop}
\aliascntresetthe{cor}

\theoremstyle{remark}
\newtheorem*{remark}{Remark}

\theoremstyle{definition}
\newtheorem{defn}{Definition}[section]
\newtheorem{example}{Example}[section]

\begin{document}
\title[Symplectic implosion and Grothendieck--Springer resolution]{Symplectic implosion and the Grothendieck--Springer resolution}
\author{Pavel Safronov}
\address{Mathematical Institute, Radcliffe Observatory Quarter, Woodstock Road, Oxford UK, OX2 6GG}
\email{safronov@maths.ox.ac.uk}
\maketitle

\begin{abstract}
We prove that the Grothendieck--Springer simultaneous resolution viewed as a correspondence between the adjoint quotient of a Lie algebra and its maximal torus is Lagrangian in the sense of shifted symplectic structures. As Hamiltonian spaces can be interpreted as Lagrangians in the adjoint quotient, this allows one to reduce a Hamiltonian $G$-space to a Hamiltonian $H$-space where $H$ is the maximal torus of $G$. We show that this procedure coincides with an algebraic version of symplectic implosion of Guillemin, Jeffrey and Sjamaar. We explain how to obtain generalizations of this picture to quasi-Hamiltonian spaces and their elliptic version.
\end{abstract}

\section*{Introduction}

The goal of this paper is to introduce symplectic implosion in the realm of derived symplectic geometry.

\subsection*{Derived symplectic geometry}

Pantev, To\"{e}n, Vaqui\'{e} and Vezzosi \cite{PTVV} introduced the notions of closed differential forms on derived stacks and defined shifted symplectic structures on such stacks. As in the classical context, a symplectic structure is a closed non-degenerate two-form on the stack, but now the form can have a nontrivial cohomological degree. Moreover, the form is not strictly closed, but closed only up to homotopy.

One can also introduce Lagrangians in a shifted symplectic stack $X$. These are morphisms $f\colon L\rightarrow X$ together with a trivialization of the pullback of the symplectic form from $X$ to $L$; moreover, we require the trivialization to be non-degenerate in a certain sense. Note that Lagrangians $L\rightarrow X$ are not necessarily embeddings: for instance, if $L\rightarrow X$ is a Lagrangian in an $n$-shifted symplectic stack $X$ and $Y$ is an $(n-1)$-shifted symplectic stack, then $L\times Y\rightarrow X$ is also Lagrangian. Moreover, in contrast to the classical setting, being a Lagrangian is not a property but an extra structure.

The key result about derived Lagrangians is the fact that their derived intersection carries a shifted symplectic structure. More precisely, if we have two Lagrangians $L_1,L_2$ in an $n$-shifted symplectic stack $X$, the derived intersection $L_1\times_X L_2$ is $(n-1)$-shifted symplectic. More generally, if $X\leftarrow L\rightarrow Y$ is a Lagrangian correspondence and $N\rightarrow Y$ is another Lagrangian, $L\times_Y N\rightarrow X$ also carries a Lagrangian structure. This should be contrasted to the case of ordinary differential geometry where the intersection of manifolds is a manifold only if the intersection is transverse.

\subsection*{Hamiltonian reduction}

Let us now explain how shifted symplectic structures give a new point of view on Hamiltonian reduction (this perspective can be found in \cite{Cal} and \cite{Saf}).

Let $\B G = [\pt / G]$ be the classifying stack of an algebraic group $G$. We can identify the coadjoint quotient \[[\g^*/G]\cong \T^*[1](\B G),\] so it carries a 1-shifted symplectic structure generalizing the classical construction of the symplectic structure on a cotangent bundle. Given a $G$-equivariant map $\mu\colon M\rightarrow \g^*$ one can ask when the induced map $\mu\colon [M/G]\rightarrow [\g^*/G]$ carries a Lagrangian structure. An easy computation shows that a Lagrangian structure on $[M/G]\rightarrow [\g^*/G]$ is the same as a closed $G$-invariant two-form on $M$ satisfying the moment map equation. In other words, Lagrangians in $[\g^*/G]$ are identified with Hamiltonian $G$-spaces, i.e. $G$-spaces $M$ together with a symplectic structure on $M$ preserved by $G$ and a $G$-equivariant moment map $\mu\colon M\rightarrow \mathfrak{g}^*$ which is a Hamiltonian for the $G$-action.

Given a Hamiltonian $G$-space its Hamiltonian reduction is defined to be
\[M//G = [\mu^{-1}(0)/G] \cong [M/G] \times_{[\g^*/G]} [\pt/G].\]

Both $[M/G]$ and $[\pt/G]$ are Lagrangians in $[\g^*/G]$, so $M//G$ is a Lagrangian intersection which, therefore, carries a symplectic structure. This recovers the standard symplectic structure on the Hamiltonian reduction when $0$ is a regular value for the moment map $\mu$ and the $G$-action on $\mu^{-1}(0)$ is free.

This picture generalizes to quasi-Hamiltonian reduction which is concerned with group-valued moment maps $M\rightarrow G$. In that case 
\[\left[\frac{G}{G}\right] = \B G \times_{\B G\times \B G} \B G\]
is the self-intersection of the diagonal in $\B G\times \B G$, a 2-shifted symplectic stack, so it carries a 1-shifted symplectic structure. Asking the same question for $G$-equivariant maps $M\rightarrow G$ we get that Lagrangians in $\left[\frac{G}{G}\right]$ are the same as quasi-Hamiltonian $G$-spaces.

One can perform Hamiltonian reduction given any two Lagrangians in a 1-shifted symplectic stack. Another example of a 1-shifted symplectic stack is $\Bun_G(E)$, the moduli stack of $G$-bundles on an elliptic curve. Such ``elliptic'' Hamiltonian reduction is useful to construct symplectic structures on the moduli stacks of $G$-bundles on del Pezzo surfaces (see \cite{Cal}).

\subsection*{Symplectic implosion}

Symplectic implosion was introduced by Guillemin, Jeffrey and Sjamaar as a way to produce Hamiltonian $T$-spaces out of Hamiltonian $K$-spaces, where $T\subset K$ is the maximal torus in a compact Lie group $K$. It is defined in a rather ad hoc way as a symplectic completion of the cross-section, which is roughly the preimage of the fundamental Weyl chamber under the moment map. It can also be interpreted as a $K$-Hamiltonian reduction with respect to $(\T^*K)_{impl}$, the universal implosion space. In \cite{GJS} it was shown that $(\T^*K)_{impl}\cong [G/N]^{\aff}$, the affinization of the base affine space, where $G$ is a complex reductive group having $K$ as a maximal compact subgroup and $N\subset G$ is the maximal unipotent subgroup.

Symplectic implosion was generalized to the hyperK\"{a}hler setting for $K=SU(n)$ by Dancer, Kirwan and Swann \cite{DKS}. As one expects hyperK\"{a}hler implosion for $K$ to coincide with the holomorphic symplectic implosion for $G$, we will use their definition to compare our results. The universal implosion $(\T^* K)_{HKimpl}$ they obtain is a hyperK\"{a}hler space which is isomorphic to $[G\times^N \b]^{\aff}$ in one of the complex structures, where $\b$ is the Borel subalgebra of $\g=\Lie(G)$. Symplectic implosion for quasi-Hamiltonian spaces was defined in \cite{HJS} for a compact group and in \cite{DK} for $SL(n; \mathbb{C})$.

Since Hamiltonian spaces are interpreted as Lagrangians in the adjoint quotient, to implode a Hamiltonian $G$-space to a Hamiltonian $H$-space one has to compose the Lagrangian with a Lagrangian correspondence between $[\g/G]$ and $[\h/H]$. One famous such correspondence is the so-called Grothendieck--Springer simultaneous resolution
\[
\xymatrix{
& [\widetilde{\g}/G] \ar[dl] \ar[dr] & \\
[\g/G] && [\h/H].
}
\]
where $\widetilde{\g}$ is the moduli space of Borel subgroups $B\subset G$ together with an element in their Lie algebras.

In this paper we show that this correspondence is Lagrangian (Corollary \ref{cor:GScorrespondence}) and, moreover, that the composition of a Lagrangian in $[\g/G]$, a Hamiltonian $G$-space, gives the holomorphic symplectic implosion. More precisely, we compute the symplectic implosion of the universal space, $\T^*G$, and show that $(\T^*G)_{impl} = G\times^N \b$ (Proposition \ref{prop:univimplosion}) which we regard as a stack instead of an affine scheme.

The Grothendieck--Springer correspondence has generalizations to the group and elliptic cases. These are correspondences
\[
\xymatrix{
& \left[\frac{\widetilde{G}}{G}\right] \ar[dl] \ar[dr] & \\
\left[\frac{G}{G}\right] && \left[\frac{H}{H}\right]
}
\]
and
\[
\xymatrix{
& \Bun_B(E) \ar[dl] \ar[dr] & \\
\Bun_G(E) && \Bun_H(E).
}
\]

Another generalization is to parabolic subgroups $P\subset G$, which in the group case is a correspondence
\[
\xymatrix{
& \left[\frac{P}{P}\right] \ar[dl] \ar[dr] & \\
\left[\frac{G}{G}\right] && \left[\frac{M}{M}\right],
}
\]
where $M$ is the Levi factor of $P$.

We show (Corollary \ref{cor:multGScorrespondence}) that all these correspondences are Lagrangian, which allows one to perform symplectic implosion in the quasi-Hamiltonian and elliptic setting. The generalization of symplectic implosion to parabolics provides an interpolation between the original unreduced space in the case $P=G$ and the usual symplectic implosion in the case $P=B$.

Having established the definition of symplectic implosion in our setting, we are able to compute $H$-Hamiltonian reduction of the imploded space in terms of the $G$-Hamiltonian reduction of the original space. Let us quote the result (Theorem \ref{thm:reductionofimplosion}) in the Hamiltonian case.

\begin{theorm}
Let $X$ be a Hamiltonian $G$-space.

The $H$-Hamiltonian reduction of the symplectic implosion $X_{impl}$ at the zero moment map value is isomorphic to the $G$-Hamiltonian reduction of $X$ with respect to the Hamiltonian $G$-space $\T^*(G/B)$, the Springer resolution of the nilpotent cone.

The $H$-Hamiltonian reduction of $X_{impl}$ at a regular semisimple moment map value $\lambda\in \h$ is isomorphic to the $G$-Hamiltonian reduction of $X$ along the adjoint orbit of $\lambda$.
\end{theorm}

The discrepancy between the Hamiltonian reduction of the implosion and the original space can be explained by noting that the procedure of implosion is not invertible. However, every Lagrangian correspondence has an adjoint and we discuss in Section \ref{sect:symplexplosion} the procedure dual to symplectic implosion (which may perhaps be named ``symplectic explosion''): it is an operation that takes Hamiltonian $H$-spaces to Hamiltonian $G$-spaces using the Grothendieck--Springer correspondence read backwards. We show (Proposition \ref{prop:implosioninvert}) that the composition of the symplectic implosion and its dual is not the identity.

\subsection*{Structure of the paper}

The paper is organized as follows. In Section \ref{sect:symplecticgeometry} we recall the necessary material on derived symplectic geometry from \cite{PTVV}. In Section \ref{sect:hamiltonianreduction} we explain how Hamiltonian and quasi-Hamiltonian reductions can be presented as Lagrangian intersections which explains the symplectic structure. Finally, Section \ref{sect:symplecticimplosion} is devoted to symplectic implosion. First, we show that various versions of the Grothendieck--Springer correspondence are Lagrangian. Second, we show that the universal implosion in the Hamiltonian case coincides with the one obtained by \cite{DKS} in the context of hyperK\"{a}hler implosion if one interprets quotients as stacky quotients instead of affine quotients. We end with a discussion of an operation dual to symplectic implosion.

\subsection*{Acknowledgements}

The author would like to thank Frances Kirwan for useful comments and a seminar talk which prompted the writing of this paper and a referee for extensive comments which improved the exposition of the paper. This research was supported by the EPSRC grant EP/I033343/1.

\section{Derived symplectic geometry}

\label{sect:symplecticgeometry}

In this section we briefly recall the necessary basics of derived symplectic geometry. The reader is invited to consult \cite{PTVV} for details and precise statements. All stacks we consider will be derived Artin stacks locally of finite presentation. In particular, the cotangent complex $\bL_X$ of such a stack $X$ is perfect and we have its dual $\bT_X=\bL_X^*$, the tangent complex.

\subsection{Symplectic structures}

Pantev, To\"{e}n, Vaqui\'{e} and Vezzosi \cite{PTVV} define the de Rham algebra $\DR(X)$ of a derived stack $X$. It is a commutative dg algebra together with an extra \emph{weight} grading and a de Rham differential $\ddr$. Our convention is such that the de Rham differential $\ddr$ has degree $1$ and weight $1$.

As a plain graded commutative dg algebra, we can identify
\begin{equation}
\DR(X)\cong \Gamma(X, \Sym(\bL_X[-1]))
\label{eq:DRcdga}
\end{equation}
where the weight grading on the right-hand side is coming from the obvious grading on the symmetric algebra.

A $\d$-closed element of $\DR(X)$ of weight $p$ and degree $p+n$ is called a \emph{$p$-form of degree $n$} on $X$. A $(\d+\ddr)$-closed element of $\DR(X)$ of weight at least $p$ and degree $p+n$ is called a \emph{closed $p$-form of degree $n$} on $X$. Explicitly, a closed $p$-form of degree $n$ is a collection of elements $\omega_0,\omega_1,...$ of $\DR(X)$ where $\omega_i$ has weight $p+i$ and degree $p+n$ which satisfy the equations
\begin{align*}
\d\omega_0 &= 0\\
\ddr\omega_i + \d\omega_{i+1} &= 0.
\end{align*}
A way to think of these equations is as saying that $\omega_i$ is not strictly $\ddr$-closed, but only closed up to homotopy given by $\omega_{i+1}$. We denote by $\Omega^p(X, n)$ the complex of $p$-forms of degree $n$ and by $\Omega^{p, cl}(X, n)$ the complex of closed $p$-forms of degree $n$. Note that the latter complex has differential given by $\d+\ddr$.

\begin{example}
Let $X = [Y/G]$ be a quotient of a smooth scheme $Y$ by a reductive algebraic group $G$. Then we can identify
\[\DR(Y)\cong \Gamma(Y, \Sym(\Omega^1_Y[-1]))\]
and
\[\DR(X)\cong (\DR(Y)\otimes \Sym(\mathfrak{g}^*[-2]))^G\]
where $\Gamma(Y, -)$ refers to the derived space of global sections.

The weight grading on $\DR(X)$ is given by placing $\mathfrak{g}^*$ in weight 1. The de Rham differential $\ddr$ on $\DR(X)$ is coming from the standard de Rham differential on $\DR(Y)$. The internal differential $\d$ on $\DR(X)$ is the sum of the internal differential on $\DR(Y)$ and the Cartan differential given on one-forms by $\alpha\mapsto -\iota_{a(-)} \alpha$ where $a\colon \mathfrak{g}\rightarrow \Gamma(Y, \T_Y)$ is the infinitesimal action map.
\label{ex:DRquotient}
\end{example}

Suppose $\omega_0$ is a two-form of degree $n$. Then contraction with vector fields gives a morphism
\[\omega_0\colon \bT_X\rightarrow \bL_X[n].\]
We say that $\omega=\omega_0+\omega_1+...$ is \textit{non-degenerate} if $\omega_0$ defines a quasi-isomorphism.

\begin{defn}
An \textit{$n$-shifted symplectic structure} on a derived stack $X$ is a closed non-degenerate two-form $\omega$ of degree $n$.
\end{defn}

We will encounter the following two basic examples of symplectic stacks.

\begin{example}
Suppose $X$ is an derived Artin stack. Then we can define the shifted cotangent stack $\T^*[n]X$ as
\[\T^*[n]X = \Spec_{\cO_X} \Sym(\bT_X[-n]).\]

In this setting one can define a Liouville one-form $\lambda$ of degree $n$ and a closed two-form $\omega = \ddr\lambda$ of the same degree $n$. A local calculation (\cite[Proposition 1.21]{PTVV}) shows that $\omega$ is an $n$-shifted symplectic structure if $X$ is a derived Deligne--Mumford stack.
\end{example}

\begin{example}
Let $G$ be an affine algebraic group. One has an equivalence of symmetric monoidal dg-categories $\QCoh(\B G)\cong \Rep G$. Under this equivalence the cotangent complex $\bL_{\B G}$ corresponds to the coadjoint representation $\mathfrak{g}^*[-1]$ placed in degree 1. Therefore, by \eqref{eq:DRcdga} one has an equivalence of graded complexes
\[\DR(\B G)\cong \Gamma(\B G, \Sym(\mathfrak{g}^*[-2]))\cong \C^\bullet(G, \Sym(\mathfrak{g}^*[-2]))\]
where $\C^\bullet(G, -)$ refers to the group cohomology cochains.

By degree reasons the de Rham differential annihilates elements of \[\Sym(\mathfrak{g}^*[-2])^G\subset \DR(\B G).\] Therefore, the space of two-forms of degree 2 on $\B G$ is equivalent to the space of closed two-forms of degree 2 which is equivalent to the set $\Sym^2(\mathfrak{g}^*)^G$.

A two-form corresponding to $c_G\in\Sym^2(\mathfrak{g}^*)^G$ is non-degenerate iff the induced map $\mathfrak{g}\rightarrow \mathfrak{g}^*$ is an isomorphism. Therefore, in this case we get a 2-shifted symplectic structure on $\B G$ (see \cite[Section 1.2]{PTVV}).
\label{ex:BGsymplectic}
\end{example}

\subsection{Lagrangians}

In ordinary symplectic geometry we say that a submanifold $L\subset X$ of a symplectic manifold is isotropic if the symplectic form restricts to zero on $L$. Since we are working in the homotopical context, the form might restrict to zero only up to homotopy.

Let $X$ be an $n$-shifted symplectic stack.

\begin{defn}
An \textit{isotropic structure} on a morphism $f\colon L\rightarrow X$ is a homotopy from $f^*\omega_X$ to $0$ in $\Omega^{2, cl}(L, n)$.
\end{defn}

Thus, an isotropic structure on a morphism $f\colon L\rightarrow X$ (not necessarily an embedding) is a collection of differential forms $h=h_0 + h_1+...$ satisfying
\[f^*\omega = (\d + \ddr) h.\]

Unpacking this definition, we see that we must have
\begin{align*}
f^*\omega_0 &= \d h_0 \\
f^*\omega_i &= \ddr h_i + \d h_{i+1}.
\end{align*}

Since $h_0$ is not $\d$-closed, it does not define a morphism of complexes $\bT_L\rightarrow \bL_L[n-1]$. Instead, consider the morphism $f^*\bT_X\rightarrow \bL_L[n]$ defined as the composition
\[f^*\bT_X\stackrel{\omega_0}\rightarrow f^*\bL_X[n]\rightarrow \bL_L[n].\]

Then $h_0$ is the null-homotopy of the composite
\[\bT_L\rightarrow f^*\bT_X\rightarrow \bL_L[n].\]

If we denote by $\bT_{L/X}$ the homotopy fiber of $\bT_L\rightarrow f^*\bT_X$, then we get a morphism of complexes $\bT_{L/X}\rightarrow \bL_L[n-1]$.

\begin{defn}
The isotropic morphism $f\colon L\rightarrow X$ is \textit{Lagrangian} if the induced morphism $\bT_{L/X}\rightarrow \bL_L[n-1]$ is a quasi-isomorphism.
\end{defn}

\begin{example}
The point $\pt$ has a unique $n$-shifted symplectic structure for any $n$. An isotropic structure on a projection $X\rightarrow \pt$ is then a closed two-form of degree $(n-1)$. It is Lagrangian iff the two-form is symplectic.
\end{example}

We have the following important theorem about Lagrangian intersections \cite[Theorem 2.9]{PTVV}.

\begin{thm}
\label{thm:lagrinter}
Suppose $L_1,L_2\rightarrow X$ are two Lagrangians into an $n$-shifted symplectic stack. Then their intersection $L_1\times_X L_2$ carries an $(n-1)$-shifted symplectic structure.
\end{thm}

The symplectic structure is constructed from the following observation: both Lagrangians carry a trivialization of the pullback of the symplectic structure on $X$. Therefore, their intersection carries two such trivializations and their difference defines an actual closed two-form.

We will need a slight generalization of this theorem (\cite[Theorem 1.2]{Saf} and \cite[Theorem 4.4]{Cal}). Given a symplectic stack $X$, we denote by $\overline{X}$ the same stack with the opposite symplectic structure.

\begin{defn}
A correspondence
\[
\xymatrix{
& L \ar[dl] \ar[dr] & \\
X && Y
}
\]
between $n$-shifted symplectic stacks $X$ and $Y$ is \textit{Lagrangian} if the morphism $L\rightarrow \overline{X}\times Y$ is Lagrangian.
\end{defn}

\begin{thm}
\label{thm:lagrcomp}
Suppose $X\leftarrow C\rightarrow Y$ is a Lagrangian correspondence and $L\rightarrow Y$ is a Lagrangian, where $X$ and $Y$ are $n$-shifted symplectic stacks. Then the intersection $C\times_Y L\rightarrow X$ is Lagrangian.
\end{thm}
The previous theorem can be recovered if we let $X=\pt$ with its canonical $n$-shifted symplectic structure.

\subsection{Examples of Lagrangians}

Let us provide some further tools to construct Lagrangians.

Recall the following classical construction. Let $Z\rightarrow X$ be an embedding of smooth manifolds. Then it is well-known that the correspondence
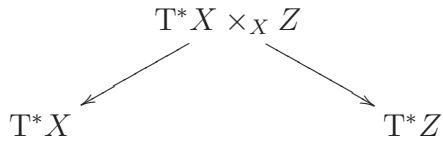
\[
\xymatrix{
& \T^*X\times_X Z\ar[dl] \ar[dr] & \\
\T^*X && \T^*Z
}
\]
is Lagrangian. The map on the left is the obvious projection and the map on the right is given by the pullback of differential forms. Let us prove an immediate generalization of this construction to shifted cotangent stacks.

\begin{prop}
Let $f\colon Z\rightarrow X$ be a morphism of derived Deligne--Mumford stacks. Then the correspondence
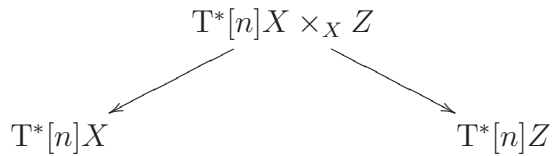
\[
\xymatrix{
& \T^*[n]X \times_X Z \ar[dl] \ar[dr] & \\
\T^*[n]X && \T^*[n]Z
}
\]
is Lagrangian.
\end{prop}
\begin{proof}
Let us recall the construction of the symplectic structure on the shifted cotangent stack $\T^*[n] Z$.

Consider the morphism
\[\cO_Z\rightarrow \bL_Z[n]\otimes \bT_Z[-n]\hookrightarrow \bL_Z[n]\otimes \Sym(\bT_Z[-n]) \cong (p_Z)_* p_Z^* (\bL_Z[n])\rightarrow (p_Z)_* \bL_{\T^*[n]Z}[n],\]
where $p_Z\colon \T^*[n]Z\rightarrow Z$ is the projection. By adjunction it gives a morphism
\[\lambda_Z\colon \cO_{\T^*[n]Z}\rightarrow \bL_{\T^*[n]Z}[n],\]
i.e. a degree $n$ one-form on $\T^*[n]Z$ known as the Liouville one-form. The symplectic structure on $\T^*[n]Z$ is defined to be $\omega_Z=\ddr\lambda_Z$.

Let $L=\T^*[n]X\times_X Z\cong \Spec \Sym_{\cO_Z}(f^*\bT_X[-n])$ and denote the maps $L\rightarrow \T^*[n] X$ and $L\rightarrow\T^*[n] Z$ by $g_X$ and $g_Z$ respectively. Let $\lambda_X$ and $\lambda_Z$ be the Liouville one-forms on $\T^*[n]X$ and $\T^*[n]Z$. The pullbacks $g_X^*\lambda_X$ and $g_Z^*\lambda_Z$ are adjoint to
\[\cO_Z\rightarrow f^*\bL_X[n]\otimes f^*\bT_X[-n]\hookrightarrow f^*\bL_X[n]\otimes \Sym(f^*\bT_X[-n])\rightarrow \bL_Z[n]\otimes \Sym(f^*\bT_X[-n])\]
and
\[\cO_Z\rightarrow \bL_Z[n]\otimes \bT_Z[-n]\hookrightarrow \bL_Z[n]\otimes \Sym(\bT_Z[-n])\rightarrow \bL_Z[n]\otimes \Sym(f^*\bT_X[-n])\]
respectively.

For any two dualizable objects $V,W$ of a symmetric monoidal 1-category $\cC$ with duality data ($\ev$,$\coev$) and a morphism $F\colon V\rightarrow W$ the diagram
\[
\xymatrix{
\textbf{1} \ar^{\coev_V}[r] \ar_{\coev_W}[d] & V^*\otimes V \ar^{\id\otimes F}[d] \\
W^*\otimes W \ar^{F^*\otimes\id}[r] & V^*\otimes W
}
\]
is commutative. Applying this to $\cC = \QCoh(Z)$, $V=\bT_Z[-n]$, $W=f^* \bT_X[-n]$ and $F\colon \bT_Z[-n]\rightarrow f^*\bT_X[-n]$ the pushforward morphism we deduce that $g_X^*\lambda_X$ and $g_Z^*\lambda_Z$ are homotopic. This gives the isotropic structure on $L\rightarrow \overline{\T^*[n]X}\times \T^*[n]Z$.

To show that it is Lagrangian, we have to check that the sequence
\begin{equation}
\label{eq:lagrangianseq}
\bT_L\rightarrow g_X^*\bT_{\T^*[n]X}\oplus g_Z^*\bT_{\T^*[n]Z} \rightarrow \bL_L[n]
\end{equation}
is a fiber sequence.

We denote the projection $L\rightarrow Z$ by $\pi$. Then we have a fiber sequence
\begin{equation}
\label{eq:fiberseq1}
\bT_{L/Z}\rightarrow \bT_L\rightarrow \pi^* \bT_Z
\end{equation}
where $\bT_{L/Z}\cong \pi^* f^*\bL_X[n]$ by the definition of $L$ as the relative spectrum of $\Sym_{\cO_Z}(f^*\bT_X[-n])$.

Pulling back similar fiber sequences for $\T^*[n]X$ and $\T^*[n]Z$ to $L$ we obtain fiber sequences
\begin{equation}
\label{eq:fiberseq2}
\xymatrix{
\pi^*f^*\bL_X[n] \ar[r] & g_X^*\bT_{\T^*[n]X} \ar[r] & \pi^*f^*\bT_X \\
\pi^*\bL_Z[n] \ar[r] & g_Z^*\bT_{\T^*[n]Z} \ar[r] & \pi^*\bT_Z.
}
\end{equation}

To show that \eqref{eq:lagrangianseq} is a fiber sequence it is enough to work \'{e}tale-locally on $X$ and $Z$, so we may assume that both $X$ and $Z$ are given by spectra of semi-free commutative dg algebras. Explicit computations in coordinates of the symplectic structures $\omega_X$ and $\omega_Z$ given in the proof of \cite[Proposition 1.21]{PTVV} then show that $\omega_X$ fits into a commutative diagram
\begin{equation}
\label{eq:cotangentsymplectic}
\xymatrix{
p_X^* \bL_X[n] \ar^{\id}[d] \ar[r] & \bT_{\T^*[n]X} \ar^{\omega_X}[d] \ar[r] & p_X^*\bT_X \ar^{\id}[d] \\
p_X^* \bL_X[n] \ar[r] & \bL_{\T^*[n]X}[n] \ar[r] & p_X^*\bT_X
}
\end{equation}
and similarly for $\omega_Z$.

Combining \eqref{eq:cotangentsymplectic} with fiber sequences \eqref{eq:fiberseq1} and \eqref{eq:fiberseq2} we obtain a commuttive diagram
\[
\xymatrix{
\pi^*f^*\bL_X[n] \ar^{\id\oplus f^*}[d] \ar[r] & \bT_L \ar[r] \ar[d] & \pi^* \bT_Z \ar^{f_*\oplus \id}[d] \\
\pi^*f^*\bL_X[n]\oplus \pi^*\bL_Z[n] \ar^{-\id\oplus \id}[d] \ar[r] & g_X^*\bT_{\T^*[n]X}\oplus g_Z^*\bT_{\T^*[n]Z} \ar^{-\omega_X\oplus\omega_Z}[d] \ar[r] & \pi^*f^*\bT_X\oplus \pi^*\bT_Z \ar^{-\id\oplus\id}[d] \\
\pi^*f^*\bL_X[n]\oplus \pi^*\bL_Z[n] \ar^{f^*\oplus \id}[d] \ar[r] & g_X^*\bL_{\T^*[n]X}[n]\oplus g_Z^*\bL_{\T^*[n]Z}[n] \ar[d]\ar[r] & \pi^*f^*\bT_X \oplus \pi^*\bT_Z \ar^{\id\oplus f_*}[d] \\
\pi^*\bL_Z[n] \ar[r] & \bL_L[n] \ar[r] & \pi^*f^*\bT_X
}
\]
where each row is a fiber sequence. Observe that in the outer columns the first, second and fourth terms form a fiber sequence. Therefore, the corresponding terms in the middle column also form a fiber sequence which shows that \eqref{eq:lagrangianseq} is a fiber sequence.
\end{proof}

This theorem gives a large family of examples of derived Lagrangians. Given a morphism $f\colon Z\rightarrow X$, the normal sheaf $N_{Z/X}$ is defined to be the cofiber of the map $\bT_Z\rightarrow f^* \bT_X$. We define the $n$-shifted conormal bundle $\N^*[n](Z/X)$ to be
\[\N^*[n](Z/X) = \Spec\Sym_{\cO_Z} (N_{Z/X}[-n]).\]

The morphism $f^*\bT_X\rightarrow N_{Z/X}$ on the level of sheaves induces a morphism of stacks
\[\N^*[n](Z/X)\rightarrow \T^*[n]X.\]

\begin{cor}
\label{cor:conormal}
Let $f\colon Z\rightarrow X$ be a morphism of derived Deligne--Mumford stacks. Then the morphism
\[\N^*[n](Z/X)\rightarrow \T^*[n]X\]
from the shifted conormal bundle of $Z$ inside $X$ is Lagrangian.
\end{cor}
\begin{proof}
We have a sequence of equivalences
\begin{align*}
(\T^*[n] X\times_X Z)\times_{\T^*[n] Z} Z &\cong \Spec \Sym_{\cO_Z} (f^*\bT_X[-n]) \times_{\Spec \Sym_{\cO_Z}(\bT_Z[-n])} \Spec \cO_Z \\
& \cong \Spec\left( \Sym_{\cO_Z} (f^*\bT_X[-n]) \otimes_{\Sym_{\cO_Z}(\bT_Z[-n])} \cO_Z\right) \\
& \cong \Spec \Sym_{\cO_Z}(N_{Z/X}[-n]) \\
& = \N^*[n](Z/X).
\end{align*}

Therefore, $\N^*[n](Z/X)\rightarrow \T^*[n]X$ can be obtained as a composition of
the zero section $Z\rightarrow \T^*[n]Z$ and the Lagrangian correspondence
\[
\xymatrix{
& \T^*[n]X \times_X Z \ar[dl] \ar[dr] & \\
\T^*[n]X && \T^*[n]Z
}
\]

By Theorem \ref{thm:lagrcomp} this implies that the morphism itself is Lagrangian.
\end{proof}

\begin{remark}
Both the theorem and the corollary remain true for derived Artin stacks if one replaces Lagrangian structures with isotropic structures. D. Calaque has recently obtained a proof that these isotropic structures are Lagrangian even in the case of derived Artin stacks.
\end{remark}

\section{Hamiltonian reduction}

\label{sect:hamiltonianreduction}

Let us present Hamiltonian and quasi-Hamiltonian reductions from the point of view of Lagrangian intersections. The details can be found in \cite{Cal} and \cite{Saf}.

\subsection{Ordinary Hamiltonian reduction}

\label{sect:hamreduction}

The stack $X=[\g^*/G]\cong \T^*[1] (\B G)$ has a 1-shifted symplectic structure which we are going to write down explicitly. The category of quasi-coherent sheaves on $X$ is equivalent to $G$-equivariant quasi-coherent sheaves on $\mathfrak{g}^*$. Under this equivalence the cotangent complex of $X$ is
\[\bL_X\cong (\g\otimes \cO_{\g^*}\rightarrow \g^*\otimes \cO_{\g^*})\]
in degrees 0 and 1 with the differential given by the coadjoint action.

From Example \ref{ex:DRquotient} we have that
\[\DR(X)\cong (\DR(\mathfrak{g}^*)\otimes \Sym(\mathfrak{g}^*[-2]))^G.\]
The Liouville one-form $\lambda_{[\g^*/G]}$ on $[\g^*/G]$ is given by the identity function $\g^*\rightarrow \g^*$ viewed as an element of $(\cO_{\g^*}\otimes \mathfrak{g}^*)^G\subset \DR(X)$ of weight $1$ and degree $2$.

We define $\omega_{[\g^*/G]} = \ddr\lambda_{[\g^*/G]}$. It is a closed two-form of degree 1 by construction. The element $\omega\in(\bL_{\g^*}\otimes \g^*)^G$ can be described as follows. Given a tangent vector to $\g^*$ at some point, $\omega$ regards it as an element of $\g^*$ using the vector space structure on $\g^*$. The symplectic structure $\omega_{[\g^*/G]}$ induces an isomorphism
\[
\xymatrix{
\g\otimes \cO_{\g^*} \ar[r] \ar^{\id}[d] & \g^*\otimes \cO_{\g^*} \ar^{\id}[d] \\
\g\otimes \cO_{\g^*}\ar[r] & \g^*\otimes \cO_{\g^*}
}
\]

Suppose we have a map $\mu\colon M\rightarrow \g^*$ from a smooth scheme $M$. One might wonder when the induced map $\mu\colon [M/G]\rightarrow [\g^*/G]$ on the quotients is isotropic or Lagrangian.

Let us recall that a Hamiltonian $G$-space $M$ is the following collection of data:
\begin{itemize}
\item A smooth 0-shifted symplectic scheme $(M,\omega)$,

\item A $G$-action on $M$ preserving the symplectic structure,

\item A $G$-equivariant moment map $\mu\colon M\rightarrow \mathfrak{g}^*$ satisfying
\[\ddr \mu(v) = \iota_{a(v)}\omega\]
for all $v\in\mathfrak{g}$ where $a\colon \mathfrak{g}\rightarrow \Gamma(M, \T_M)$ is the infinitesimal action map.
\end{itemize}

The following theorem was proved in \cite[Section 2.2.1]{Cal} and \cite[Section 2.2]{Saf}:
\begin{thm}
Let $M$ be a smooth scheme with a $G$-action. Then the data of a Hamiltonian $G$-space on $\mu\colon M\rightarrow \g^*$ is equivalent to a Lagrangian structure on $\overline{\mu}\colon [M/G]\rightarrow [\g^*/G]$.
\end{thm}

More generally, one can think of Lagrangians $L\rightarrow [\g^*/G]$ as derived Hamiltonian $G$-spaces. Given such a Lagrangian we have the underlying symplectic stack given by
\[L_{symp} = L\times_{[\g^*/G]} \g^*\]
which carries a natural $G$-action and the reduction
\[L_{red} = L\times_{[\g^*/G]} [\pt/G]\]
where $[\pt/G]\rightarrow [\g^*/G]$ is given by the inclusion of the origin.

By Theorem \ref{thm:lagrinter} both of these carry a 0-shifted symplectic structure. Moreover, if $M$ is a Hamiltonian $G$-space, then
\[[M/G]_{symp}\cong M\]
and
\[[M/G]_{red} \cong M//G = [\mu^{-1}(0)/G]\]
is the symplectic reduction of $M$.

Let us give two examples of Hamiltonian $G$-spaces.

\begin{enumerate}
\item Recall that the cotangent bundle of a $G$-manifold $X$ is naturally a Hamiltonian $G$-space and the Hamiltonian reduction $\T^* X // G$ recovers $\T^*(X/G)$. The same construction works on the derived level as well.

Let $X$ be a derived Deligne--Mumford stack with a $G$-action such that $[X/G]$ is also Deligne--Mumford. There is a morphism $Y=[X/G]\rightarrow \B G$. The 1-shifted conormal bundle $\N^*[1](Y/\B G)\rightarrow \T^*[1](\B G)$ is Lagrangian by Corollary \ref{cor:conormal}. The Hamiltonian reduction $\N^*[1](Y/\B G)_{red}$ is given by a composition of the Lagrangian correspondences
\[
\xymatrix{
& \B G \ar[dl] \ar[dr] && \T^*[1](\B G)\times_{\B G} Y \ar[dl] \ar[dr] && Y \ar[dl] \ar[dr] & \\
\pt && \T^*[1](\B G) && \T^*[1]Y && \pt
}
\]
where the composition of the two correspondences on the right gives \[N^*[1](Y/\B G)\rightarrow \T^*[1](\B G).\]

To relate it to the classical construction note that the shifted conormal complex $\N^*_{Y/\B G}[1]$ is the same as the relative cotangent complex $\bL_{Y/\B G}$. But the morphism $[\T^*X/G]\rightarrow [X/G]$ can be identified with the total space of the bundle $\bL_{Y/\B G}$. Therefore,
\[\N^*[1](Y/\B G)\cong [\T^*X / G]\]
and so $\N^*[1](Y/\B G)_{symp}\cong \T^* X$.

This construction remains valid for derived Artin stacks if one doesn't require non-degeneracy of the two-forms involved.

\item Consider a coadjoint orbit $\cO\subset \g^*$. Let $G_\cO$ be the stabilizer of a point in $\cO$. Then the map
\[[\cO/G]\cong [\pt/G_\cO]\rightarrow [\g^*/G]\]
is isotropic since a two-form of degree 1 on $[\pt/G_\cO]$ is necessarily zero. An easy check shows that the zero isotropic structure is in fact Lagrangian. The isotropic structure gives a two-form on $\cO$ which is nothing else but the Kirillov--Kostant--Souriau symplectic structure on a coadjoint orbit.

We define the Hamiltonian reduction of $M$ with respect to $G$ along a coadjoint orbit $\cO$ to be
\[M//_{\cO}G = [\mu^{-1}(\cO) / G] = [(M\times_{\g^*} \pt)/G] \cong [M/G] \times_{[\g^*/G]} [\cO/G].\]

It is again a Lagrangian intersection, so it carries a symplectic structure.
\end{enumerate}

\subsection{Quasi-Hamiltonian reduction}
\label{sect:qhamreduction}

In this section we assume $G$ is a reductive algebraic group. Choose a $G$-invariant non-degenerate bilinear form $c_G$ on $\g$ that we denote by $(-, -)$. By Example \ref{ex:BGsymplectic} it gives a 2-shifted symplectic structure on the classifying stack $\B G$. Therefore,
\[\left[\frac{G}{G}\right]\cong \B G \times_{\B G\times\overline{\B G}} \B G,\]
a self-intersection of the diagonal $\B G$, is a Lagrangian intersection and hence it carries a natural 1-shifted symplectic structure. Here and in the future the horizontal line denotes the adjoint quotient.

In \cite{Saf} we showed that this 1-shifted symplectic structure on $\left[\frac{G}{G}\right]$ has the following description. By Example \ref{ex:DRquotient} we have
\[\DR\left(\left[\frac{G}{G}\right]\right) = (\DR(G)\otimes \Sym(\mathfrak{g}^*[-2]))^G.\]

We have a two-form of degree 1
\[\omega_0 = -\frac{1}{2}(\theta+\overline{\theta}, -)\]
and a three-form of degree 0
\[\omega_1 = \frac{1}{12}(\theta, [\theta, \theta]),\]
where $\theta$ and $\overline{\theta}$ are the Maurer--Cartan forms in $\Omega^1(G)\otimes \mathfrak{g}$. The symplectic structure on $\left[\frac{G}{G}\right]$ is given by $\omega_0 + \omega_1$.

Let us recall \cite[Definition 2.2]{AMM} that a \emph{quasi-Hamiltonian $G$-space} is the following collection of data:
\begin{itemize}
\item A smooth scheme $M$ with a two-form $\omega$,
\item A $G$-action on $M$ preserving $\omega$,
\item A $G$-equivariant moment map $\mu\colon M\rightarrow G$ satisfying
\begin{align*}
\ddr\omega &= \mu^*\omega_1 \\
\iota_{a(v)}\omega &= -\mu^*\omega_0(v)
\end{align*}
for every $v\in\mathfrak{g}$.
\end{itemize}
Moreover, we require the following non-degeneracy condition: for every $x\in M$ we have
\[\ker(\omega_x) = \{a(v)\ |\ \Ad_{\mu(x)} v = -v\}.\]

\begin{thm}
The data of a quasi-Hamiltonian $G$-space $\mu\colon M\rightarrow G$ is equivalent to a Lagrangian structure on \[[M/G]\rightarrow \left[\frac{G}{G}\right].\]
\end{thm}

As before, given a Lagrangian $L\rightarrow \left[\frac{G}{G}\right]$ we can regard it as a generalized quasi-Hamiltonian $G$-space. Given such a Lagrangian, we define its reduction to be
\[L_{red} = L\times_{\left[\frac{G}{G}\right]} [\pt/G]\]
with $[\pt/G]\hookrightarrow \left[\frac{G}{G}\right]$ the inclusion of the unit element. It carries a 0-shifted symplectic structure as an intersection of two Lagrangians. Note that
\[L_{qsymp} = L\times_{\left[\frac{G}{G}\right]} G\]
is not symplectic since $G\rightarrow \left[\frac{G}{G}\right]$ is not Lagrangian.

If $L=[M/G]$ for a quasi-Hamiltonian $G$-space $M$, then we have $L_{qsymp}\cong M$ and
\[L_{red} \cong [\mu^{-1}(e)/G]=M//G\]
the usual quasi-Hamiltonian reduction of $M$.

More generally, a conjugacy class $\cO\subset G$ gives a Lagrangian $[\frac{\cO}{G}]\rightarrow[\frac{G}{G}]$ and we define the quasi-Hamiltonian reduction of $M$ with respect to $G$ along $\cO$ to be the Lagrangian intersection
\[M//_\cO G = [\mu^{-1}(\cO) / G] = [(M\times_G \cO) / G] \cong [M/G]\times_{\left[\frac{G}{G}\right]} [\cO/G].\]

\section{Symplectic implosion}

\label{sect:symplecticimplosion}

In this section $G$ denotes a split connected reductive group over a characteristic zero field $k$.

\subsection{Grothendieck--Springer resolution}

Let $B\subset G$ be a Borel subgroup and $p\colon B\twoheadrightarrow H$ the abelianization map; we denote by $\b$ and $\h$ the corresponding Lie algebras. The kernel of $p$ is denoted by $N$ whose Lie algebra is denoted by $\n$. The constructions we are about to describe can be written in a way independent of the choice of the Borel, but we choose it for the sake of exposition.

One defines the Grothendieck--Springer simultaneous resolution $\widetilde{\g}$ to be the vector bundle \[\widetilde{\g}=G\times^B \b\]
over the flag variety $G/B$, see \cite[Section 3.1.31]{CG}.

We have a map $\widetilde{\g}\rightarrow \g$ given by
\[(g, x)\mapsto \Ad_g(x)\]
and $\widetilde{\g}$ can be described as the space of elements $x$ of $\g$ together with a choice of a Borel containing $x$.

There is a $G$-action on $\widetilde{\g}$ given by the left action on $G$. This makes $\widetilde{\g}\rightarrow \g$ into a $G$-equivariant map.

We also have a map $\widetilde{\g}\rightarrow \h$ given by the composition
\[G\times^B \b\rightarrow G\times^B \h\rightarrow \h\]
using the fact that $B$ acts trivially on $\h$.

Combining all these maps we get a correspondence
\begin{equation}
\label{eq:GScorrespondence}
\xymatrix{
& [\widetilde{\g}/G] \ar[dl] \ar[dr] & \\
[\g/G] && [\h/H]
}
\end{equation}
Note that $[\widetilde{\g}/G] \cong [\b/B]$.

Similarly, there is a group version of the Grothendieck--Springer resolution given by
\[\widetilde{G} = G\times^B B,\]
where $B$ acts on itself by conjugation. This gives a correspondence
\begin{equation}
\label{eq:multGScorrespondence}
\xymatrix{
& \left[\frac{\widetilde{G}}{G}\right] \ar[dl] \ar[dr] & \\
\left[\frac{G}{G}\right] && \left[\frac{H}{H}\right]
}
\end{equation}
where again \[\left[\frac{\widetilde{G}}{G}\right] \cong \left[\frac{B}{B}\right].\]

\subsection{Lagrangian structure}

In this section we will slightly generalize the discussion, so choose a parabolic subgroup $P\subset G$ with Levi factor $M$. We denote the corresponding Lie algebras by $\p$ and $\m$. The reader may assume that $P=B$ and $M=H$.

Pick a $G$-invariant non-degenerate symmetric bilinear pairing $c_G\in \Sym^2(\g^*)^G$. By restriction we get a bilinear pairing in $\Sym^2(\p^*)^P$. Similarly, we have a pullback morphism
\[\Sym^2(\m^*)^M\rightarrow \Sym^2(\p^*)^P.\]

\begin{prop}
\label{prop:parabolicpairing}
The morphism
\[\Sym^2(\m^*)^M\rightarrow \Sym^2(\p^*)^P\]
is an isomorphism.

Moreover, the composition
\[\Sym^2(\g^*)^G\rightarrow \Sym^2(\p^*)^P\cong \Sym^2(\m^*)^M\]
sends non-degenerate pairings to non-degenerate pairings.
\end{prop}
\begin{proof}
We denote by $\u$ the Lie algebra of the unipotent radical of $P$, so $\p\cong \u\oplus \m$. Without loss of generality, we may assume that $P$ is a standard parabolic subgroup, i.e. it contains our chosen Borel $B$. Choosing a splitting $H\subset B$ we get the root decomposition
\[\g\cong \h\oplus \bigoplus_{\alpha\in\Phi} \g_\alpha\]
where $\Phi$ is the set of roots of $\g$. We denote by $\Phi^+$ the set of positive roots with respect to $B$ and $e_\alpha\in\g_\alpha$ basis elements.

We have the following description of standard parabolics (see \cite[Proposition 14.18]{Bo}). Fix a subset $I\subset \Delta$ of simple roots and denote by $[I]$ the root subsystem generated by the roots in $I$. We also denote $\Phi(I)^+=\Phi^+-[I]$, the set of positive roots of $\g$ not lying in $[I]$. Then we have isomorphisms
\begin{align*}
\u&\cong \bigoplus_{\alpha\in\Phi(I)^+} \g_\alpha, \\
\m&\cong \h\oplus\bigoplus_{\alpha\in[I]} \g_\alpha.
\end{align*}

The morphism $\Sym^2(\m^*)^M\rightarrow \Sym^2(\p^*)^P$ is clearly injective, so we have to prove it is surjective. For this it is enough to show that any $P$-invariant symmetric bilinear pairing $(-, -)$ on $\p$ vanishes on $\u$. Indeed, we have
\[0 = ([h, e_\alpha], e_\beta) + (e_\alpha, [h, e_\beta]) = (\alpha(h)+\beta(h))(e_\alpha, e_\beta)\]
for any $h\in\h$ and $\alpha,\beta$ two roots. If $\alpha\in\Phi(I)^+$ and $\beta\in[I]$, then $\alpha+\beta$ is never zero, so $(e_\alpha, e_\beta)=0$. One similarly shows that $(e_\alpha, h) = 0$ for any $h\in\h$ and $\alpha\in\Phi(I)^+$.

Finally, suppose a pairing in $\Sym^2(\g^*)^G$ is non-degenerate and consider $x\in\h\subset \m$. Since $c_G$ is non-degenerate, there is a $y\in\g$ such that $(x, y)\neq 0$. But since $c_G$ is $H$-invariant, $y$ is necessarily in $\h\subset \m$. Next, consider $e_\alpha\in\g_\alpha\subset\m$ for $\alpha\in [I]$. By non-degeneracy of $c_G$ there is a $y\in\g$ such that $(e_\alpha,y)\neq 0$. Again using $H$-invariance of the pairing we deduce that $y\in\g_{-\alpha}\subset \m$ which proves the non-degeneracy of the pairing $(-, -)$ restricted to $\m$.
\end{proof}

We denote by $c_M\in\Sym^2(\m^*)^M$ the image of $c_G$ under
\[\Sym^2(\g^*)^G\rightarrow \Sym^2(\p^*)^P\cong \Sym^2(\m^*)^M\]
which is non-degenerate by the previous Proposition.

The choice of $c_G$ allows us to identify $[\g/G]\cong [\g^*/G]$ and $[\m/M]\cong [\m^*/M]$ and therefore by Section \ref{sect:hamreduction} we obtain 1-shifted symplectic structures on $[\g/G]$ and $[\m/M]$. Let us prove the following statement.

\begin{thm}
The correspondence
\begin{equation}
\label{eq:parabolicGScorrespondence}
\xymatrix{
& [\p/P] \ar[dl] \ar[dr] & \\
[\g/G] && [\m/M]
}
\end{equation}
is Lagrangian.
\end{thm}
\begin{proof}
By construction the 1-shifted symplectic structures on $[\g/G]$ and $[\m/M]$ are exact. The primitive for the 1-shifted symplectic structure on $[\g/G]$ is given by
\[c_G\in \Sym^2(\mathfrak{g}^*)^G\subset (\cO(\g)\otimes \g^*)^G\subset (\DR(\g)\otimes \Sym(\g^*[-2]))^G[2]\]
and similarly for $[\m/M]$. By assumption both $c_G$ and $c_M$ restrict to the same element of $\Sym^2(\p^*)^P$ which gives the isotropic structure on the correspondence \eqref{eq:parabolicGScorrespondence}.

Denote $L=[\p/P]$ and $X=\overline{[\g/G]}\times [\m/M]$. To simplify the notation, let's denote the trivial vector bundle with fiber $V$ by $\underline{V}$ when the base space is clear.

Then the tangent complex $\bT_L$ is
\[\bT_L=\underline{\p}[1] \oplus \underline{\p}\]
with the differential given by the adjoint action.

To show that the isotropic structure on $L\rightarrow X$ is Lagrangian we have to prove that
\[\underline{\p}[1]\oplus \underline{\p}\rightarrow \underline{\g}[1]\oplus \underline{\g}\oplus \underline{\m}[1]\oplus \underline{\m}\rightarrow \underline{\p}^*[1]\oplus \underline{\p}^*\]
is a fiber sequence of quasi-coherent sheaves on $[\m/M]$ where the second morphism is given by composing $-c_G$ and $c_M$ with the restriction morphisms $\g^*\rightarrow \p^*$ and $\m^*\rightarrow \p^*$. For this it is enough to prove that
\begin{equation}
\label{eq:parsequence}
0\rightarrow \p\rightarrow \g\oplus \m\rightarrow \p^*\rightarrow 0
\end{equation}
is an exact sequence of vector spaces.

Clearly, the sequence is exact at the first and third terms. The Euler characteristic of the sequence is
\[2\dim \p - \dim\g - \dim\m = \dim \m+2\dim\u - \dim\g = 0,\]
which coincides with the dimension of the cohomology of the middle term, which is, therefore, also zero.
\end{proof}
\begin{cor}
\label{cor:GScorrespondence}
The Grothendieck--Springer correspondence \eqref{eq:GScorrespondence}
\[
\xymatrix{
& [\widetilde{\g}/G] \ar[dl] \ar[dr] & \\
[\g/G] && [\h/H]
}
\]
is Lagrangian.
\end{cor}

Similarly, the group version of the Grothendieck--Springer correspondence \eqref{eq:multGScorrespondence} is also Lagrangian. To show this, we need a lemma.

Recall that the choice of $c_G$ gave a 2-shifted symplectic structure on $\B G$ by Example \ref{ex:BGsymplectic}. Its restriction $c_M$ to $M$ is also non-degenerate, so defines a 2-shifted symplectic structure on $\B M$.

\begin{lm}
\label{lm:BLagrangian}
The correspondence
\[
\xymatrix{
& \B P \ar[dl] \ar[dr] & \\
\B G && \B M
}
\]
is Lagrangian. Moreover, the space of Lagrangian structures is contractible.
\end{lm}
\begin{proof}
By \eqref{eq:DRcdga} we have an identification of graded complexes
\[\DR(\B P)\cong \C^\bullet(P, \Sym(\p^*[-2])).\]

By assumption the symplectic structures on $\B G$ and $\B M$ determined by $c_G\in\Sym^2(\g^*)^G$ and $c_M\in\Sym^2(\m^*)^M$ pull back to the same element $\Sym^2(\p^*)^P\in\DR(\B P)$ which gives an isotropic structure on the correspondence. The space of such isotropic structures is a torsor over the space of closed two-forms on $\B P$ of degree 1. From the explicit identification of $\DR(\B P)$ above we see that every such form is zero.

To prove that this isotropic structure is Lagrangian we have to show that
\[\bT_{\B P}\rightarrow f^* \bT_{\B G\times \B M}\rightarrow \bL_{\B P}[2]\]
is a fiber sequence where $f\colon \B P\rightarrow \B G\times \B M$.

In other words, we have to show that the sequence of $P$-representations
\[0\rightarrow \p\rightarrow \g\oplus \m\rightarrow \p^*\rightarrow 0\]
is exact which we have already checked in the course of the proof of the previous theorem (see \eqref{eq:parsequence}).
\end{proof}

We have $\Map(S^1_\B, \B G)\cong \left[\frac{G}{G}\right]$, so now we have to show that the functor $\Map(S^1_\B, -)$ sends Lagrangian morphisms to Lagrangian morphisms. This follows from the AKSZ formalism which we briefly recall.

Recall the notion of $\cO$-compact stacks and $\cO$-orientation on such stacks \cite[Section 2.1]{PTVV}. For instance, given a topological space $Z$ we can regard it as a constant derived stack $Z_\B$. If $Z$ is a finite CW complex, the derived stack $Z_\B$ is $\cO$-compact. Moreover, if $Z$ is a closed $d$-dimensional manifold, the derived stack $Z_\B$ has an $\cO$-orientation of degree $d$.

\begin{thm}
Let $L\rightarrow X$ be a Lagrangian morphism to an $n$-shifted symplectic stack. Let $Y$ be an $\cO$-compact stack equipped with an $\cO$-orientation of degree $d$. Then the morphism
\[\Map(Y, L)\rightarrow \Map(Y, X)\]
is a Lagrangian morphism to an $(n-d)$-shifted symplectic stack.
\end{thm}

The proof of this theorem is identical to the proof of the AKSZ theorem \cite[Theorem 2.5]{PTVV} (see also \cite[Theorem 2.10]{Cal}), so we omit it. Let us present two corollaries.

Besides $S^1_\B$, a natural example of an $\cO$-compact stack equipped with an $\cO$-orientation of degree $1$ is an elliptic curve $E$ equipped with a trivialization of the canonical bundle. We denote by
\[\Bun_G(E)=\Map(E, \B G)\]
the moduli stack of $G$-bundles on the elliptic curve $E$. By \cite[Theorem 2.5]{PTVV} both $\Bun_G(E)$ and $\Bun_M(E)$ carry a 1-shifted symplectic structure.

\begin{cor}
\label{cor:multGScorrespondence}
The correspondences
\[
\xymatrix{
& \left[\frac{P}{P}\right] \ar[dl] \ar[dr] & \\
\left[\frac{G}{G}\right] && \left[\frac{M}{M}\right]
}
\]
and
\[
\xymatrix{
& \Bun_P(E) \ar[dl]\ar[dr] & \\
\Bun_G(E) && \Bun_M(E)
}
\]
are Lagrangian.
\end{cor}

Both statements are obtained by applying $\Map(S^1_\B, -)$ and $\Map(E, -)$ to the correspondence in Lemma \ref{lm:BLagrangian}.

\begin{cor}
The group version of the Grothendieck--Springer correspondence \eqref{eq:multGScorrespondence}
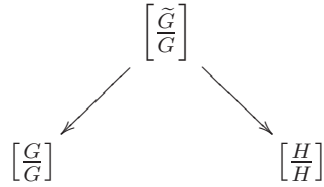
\[
\xymatrix{
& \left[\frac{\widetilde{G}}{G}\right] \ar[dl] \ar[dr] & \\
\left[\frac{G}{G}\right] && \left[\frac{H}{H}\right]
}
\]
is Lagrangian.
\end{cor}

\subsection{Symplectic implosion}

Recall from Section \ref{sect:hamreduction} that one can interpret Lagrangians in $[\g/G]$ as Hamiltonian $G$-spaces. More precisely, given a Lagrangian $L\rightarrow [\g/G]$, the space $L_{symp} = L\times_{[\g/G]} \g$ is a Hamiltonian $G$-space if it is a smooth scheme.

Since the Grothendieck--Springer correspondence \eqref{eq:GScorrespondence} is Lagrangian, we can use it and Theorem \ref{thm:lagrcomp} to turn Lagrangians in $[\g/G]$ into Lagrangians in $[\h/H]$ which are Hamiltonian $H$-spaces. If $X$ is a Hamiltonian $G$-space, the composition of the Lagrangian $[X/G]\rightarrow [\g/G]$ and the Grothendieck--Springer correspondence is given by
\[[X/G]\times_{[\g/G]} [\b/B]\cong [([X/G]\times_{[\g/G]} [\b/N])/H].\]

Thus, $[X/G]\times_{[\g/G]} [\b/N]$ is a Hamiltonian $H$-space.

\begin{defn}
\label{defn:symplimplosion}
The \emph{symplectic implosion} of a Hamiltonian $G$-space $X$ is the Hamiltonian $H$-space
\[X_{impl} = [X/G]\times_{[\g/G]} [\b/N].\]
\end{defn}
The moment map $X_{impl}\rightarrow \h\cong\h^*$ is given by the projection map $[\b/N]\rightarrow \h$ on the second factor. We recall the classical construction of symplectic implosion and its relation to our definition in Section \ref{sect:univimplosion}.

The $H$-Hamiltonian reduction of the implosion is related to the $G$-Hamiltonian reduction of the original space. Indeed, consider an inclusion $[\pt/H]\subset [\h/H]$ of an element of $\h$ which is always Lagrangian. Then we have
\begin{align*}
[X_{impl}/H] \times_{[\h/H]} [\pt/H] &= ([X/G]\times_{[\g/G]} [\b/B]) \times_{[\h/H]} [\pt/H] \\
&\cong [X/G]\times_{[\g/G]} ([\b/B] \times_{[\h/H]} [\pt/H]),
\end{align*}
where we regard
\[[\b/B]\times_{[\h/H]} [\pt/H]\]
as a Lagrangian in $[\g/G]$. Let us compute it in two opposite cases.

\begin{enumerate}
\item If $\pt\hookrightarrow \h$ is the inclusion of the origin, then \[[\b/B] \times_{[\h/H]} [\pt/H]\cong [\n/B],\]
where we recall that $\n\subset \b$ is the kernel of the projection $\b\rightarrow \h$. 

Let us recall the nilpotent cone $\cN\subset \g$ of nilpotent elements of $\g$ \cite[Section 3.2]{CG}. Its preimage under the Grothendieck--Springer resolution $\widetilde{\g}\rightarrow \g$ is the so-called Springer resolution $\widetilde{\cN}$ and can be identified with $\T^*(G/B)$, the cotangent bundle of the flag variety \cite[Lemma 3.2.2]{CG}. We have the following commutative diagram
\[
\xymatrix{
[\b/B] \ar^{\sim}[d] & [\n/B] \ar^{\sim}[d] \ar@{_{(}->}[l] \\
[\tilde{\g}/G] \ar[d] & [\widetilde{\cN}/G] \ar[d] \ar@{_{(}->}[l] \\
[\g/G] & [\cN/G] \ar@{_{(}->}[l]
}
\]

In particular, $[\n/B]_{symp}\cong \widetilde{\cN}\cong \T^*(G/B)$ is a Hamiltonian $G$-space where the Hamiltonian structure is induced from the $G$-action of $G/B$.

\item If $\pt\hookrightarrow \h$ is the inclusion of a regular semisimple element $\lambda\in\h\subset \g$, then
\[[\b/B]\times_{[\h/H]} [\pt/H] \cong [\pt/H],\]
since any regular semisimple element of $\b$ is $B$-conjugate to an element of $H$. The underlying Hamiltonian $G$-space of $[\pt/H]$ is the adjoint orbit of $\lambda\in\h\subset \g$.
\end{enumerate}

We get the following statement.

\begin{thm}
\label{thm:reductionofimplosion}
The $H$-Hamiltonian reduction of the symplectic implosion $X_{impl}$ at the zero moment map value is isomorphic to the $G$-Hamiltonian reduction of $X$ with respect to the Hamiltonian $G$-space $\T^*(G/B)$.

The $H$-Hamiltonian reduction of $X_{impl}$ at a regular semisimple moment map value $\lambda\in \h$ is isomorphic to the $G$-Hamiltonian reduction of $X$ along the adjoint orbit of $\lambda$.
\end{thm}

\subsection{Some generalizations}

The definition of symplectic implosion (Definition \ref{defn:symplimplosion}) admits an immediate generalization to the quasi-Hamiltonian case since we have a similar Lagrangian correspondence there as well.

\begin{defn}
The \textit{group-valued symplectic implosion} of a quasi-Hamiltonian $G$-space $X$ is a quasi-Hamiltonian $H$-space
\[X_{qimpl} = [X/G]\times_{\left[\frac{G}{G}\right]} \left[\frac{B}{N}\right].\]
\end{defn}

The relation between $H$-quasi-Hamiltonian reduction of implosion and $G$-quasi-Hamiltonian reduction of the original space is similar to the Lie algebra case, so let us just state the result.

Let $\cN_G\subset G$ be the variety of unipotent elements of $G$. Its pullback $\widetilde{\cN}_G$ under the group version of the Grothendieck--Springer resolution $\widetilde{G}\rightarrow G$ can be identified with
\[\widetilde{\cN}_G\cong G\times^B N\]
where $B$ acts on $N$ by conjugation. This generalizes the Lie algebra case where \[\widetilde{\cN}\cong G\times^B \n\cong \T^*(G/B).\]

Since
\[\left[\frac{N}{B}\right]\cong [\pt/H]\times_{\left[\frac{H}{H}\right]} \left[\frac{B}{B}\right]\]
is a Lagrangian intersection, the projection $\left[\frac{N}{B}\right]\rightarrow \left[\frac{G}{G}\right]$ is Lagrangian. Therefore,
\[\widetilde{\cN}_G \cong \left[\frac{N}{B}\right]\times_{\left[\frac{G}{G}\right]} G,\] carries a natural quasi-Hamiltonian $G$-structure as described in Section \ref{sect:qhamreduction}.

\begin{thm}
The $H$-quasi-Hamiltonian reduction of the group-valued symplectic implosion $X_{qimpl}$ at the unit moment map value is isomorphic to the $G$-quasi-Hamiltonian reduction of $X$ with respect to the quasi-Hamiltonian $G$-space $\widetilde{\cN}_G$.

The $H$-quasi-Hamiltonian reduction of $X_{qimpl}$ at a regular semisimple moment map value $\lambda\in H$ is isomorphic to the $G$-quasi-Hamiltonian reduction of $X$ along the $G$-conjugacy class of $\lambda$.
\end{thm}

Let us also give a version of implosion with more general parabolic subgroups. We will only give it in the group-valued case, the Lie algebra case is identical. Let $P\subset G$ be a parabolic subgroup with $U\subset P$ the unipotent radical and $M=P/U$ the Levi factor.

\begin{defn}
The \textit{partial group-valued symplectic implosion} of a quasi-Hamiltonian $G$-space $X$ is a quasi-Hamiltonian $M$-space
\[X_{qimpl} = [X/G]\times_{\left[\frac{G}{G}\right]} \left[\frac{P}{U}\right].\]
\end{defn}

For instance, if $P=G$, the implosion is isomorphic to $X$ again, so partial symplectic implosions interpolate between the original quasi-Hamiltonian space $X$ in the case $P=G$ and the imploded space $X_{qimpl}$ in the case $P=B$, a Borel subgroup.

\subsection{Universal implosion}
\label{sect:univimplosion}

In this section we relate our definition of symplectic implosion with the one present in the literature (see \cite{GJS} for the case of real symplectic manifolds and \cite{DKS} for the case of holomorphic symplectic manifolds).

Let $G=SL_n(\bC)$, $B\subset G$ the subgroup of upper-triangular matrices, $\b$ its Lie algebra and $H\subset G$ the subgroup of diagonal matrices. Dancer--Kirwan--Swann \cite{DKS} show that the space $Q=[G\times^B \b]^{\aff}$, the affinization of the stack $G\times^B \b$, carries a stratified holomorphic symplectic structure together with a $G\times H$-action making it a stratified Hamiltonian $(G\times H)$-space. This space is known as the \emph{universal implosion} space for the following reason. Given a Hamiltonian $G$-space $X$ \cite{DKS} define the holomorphic symplectic implosion of $X$ to be the $G$-Hamiltonian reduction of $X\times Q$. It carries a residual $H$-action and is, moreover, a Hamiltonian $H$-space.

Let us give a similar description of the symplectic implosion as given by Definition \ref{defn:symplimplosion}. We return to the general setting of a split connected reductive group $G$.

The group $G$ has two commuting $G$ actions given by the left and right action, so $\T^* G$ is a Hamiltonian $(G\times G)$-space. For any Hamiltonian $G$-space we have
\begin{equation}
\label{eq:cotangentid}
(X\times \T^*G) // G\cong X,
\end{equation}
so $\T^* G$ acts as a kind of identity.

As both symplectic implosion and Hamiltonian reduction are fiber products, the operations commute. Therefore, we have
\[X_{impl}\cong ((X\times \T^*G)//G)_{impl}\cong (X\times (\T^*G)_{impl})//G,\]
where $(\T^*G)_{impl}$ is a Hamiltonian $(G\times H)$-space.

\begin{prop}
\label{prop:univimplosion}
We have an isomorphism of $(G\times H)$-spaces
\[(\T^*G)_{impl}\cong G\times^B \b.\]
\end{prop}
\begin{proof}
Corollary \ref{cor:GScorrespondence} gives a Lagrangian morphism
\[[\b/B]\rightarrow \overline{[\g/G]}\times [\h/H],\]
so let's find the corresponding Hamiltonian $(G\times H)$-space.

First, pulling back this morphism along the universal bundle $\h\rightarrow [\h/H]$ we get a Cartesian square
\[
\xymatrix{
[\b/N] \ar[r] \ar[d] & [\g/G]\times \h \ar[d] \\
[\b/B] \ar[r] & [\g/G]\times [\h/H]
}
\]

Pulling back this morphism along $\g\rightarrow [\g/G]$ we get a Cartesian square
\[
\xymatrix{
G\times^N \b \ar[r] \ar[d] & \g\times \h \ar[d] \\
[\b/N] \ar[r] & [\g/G]\times \h
}
\]
In particular, $[\b/B]_{symp}\cong G\times^N \b$.

Thus $G\times^N \b$ is the Hamiltonian $(G\times H)$-space satisfying
\[X_{impl} \cong (X\times G\times^N\b) // G\]
for any Hamiltonian $G$-space $X$. Substituting $X=\T^*G$ and using equation \eqref{eq:cotangentid} we conclude that
\[(\T^* G)_{impl} \cong G\times^N\b.\]
\end{proof}

One can give the following modular interpretation of $(\T^*G)_{impl}$: it parametrizes Borel subgroups $B\subset G$ together with an element in $\Lie(B)$ and an element in $B/[B, B]$.

Similarly, in the group case we have a quasi-Hamiltonian $(G\times G)$-space $G\times G$ with the property that
\[(X\times G\times G)// G\cong X\]
for any quasi-Hamiltonian space $X$. The universal group-valued implosion is then
\[(G\times G)_{qimpl} = G\times^N B.\]

Again, its affinization for $G=SL_n(\bC)$ is the universal group-valued implosion space that Dancer and Kirwan use in \cite{DK} to define symplectic implosion for quasi-Hamiltonian spaces.

\subsection{Dual implosion}
\label{sect:symplexplosion}

In this section we define a procedure dual to that of implosion, i.e. a way to go from Hamiltonian $H$-spaces to Hamiltonian $G$-spaces. This procedure turns out to be adjoint in a precise sense to symplectic implosion, but not its inverse.

Calaque \cite[Section 4.2.2]{Cal} has defined a symmetric monoidal 1-category of Lagrangian correspondences $\LagrCorr^n_1$. Its objects are $n$-shifted symplectic stacks and morphisms from $X$ to $Y$ are given by Lagrangian correspondences $X\leftarrow L\rightarrow Y$. Haugseng \cite[Section 11]{Ha} extended this definition to a symmetric monoidal $(\infty, 1)$-category $\LagrCorr^n_{(\infty, 1)}$ of Lagrangian correspondences and showed that it has duals. Moreover, he defined a symmetric monoidal $(\infty, m)$-category $\IsotCorr^n_{(\infty, m)}$ for any $m$ of isotropic correspondences whose objects are $n$-shifted symplectic stacks, morphisms from $X$ to $Y$ are given by isotropic morphisms $L\rightarrow \overline{X}\times Y$ and higher morphisms are given by iterated correspondences. He furthermore showed that $\IsotCorr^n_{(\infty, m)}$ has duals.

Let us give an explicit description of the duals and adjoints in $\IsotCorr^n_{(\infty, m)}$ (see \cite[Lemma 9.3]{Ha}). Given an $n$-shifted symplectic stack $X$ its dual is the opposite symplectic stack $\overline{X}$ with the duality data given by the diagonal Lagrangian. Given a morphism $X\leftarrow L\rightarrow Y$ in $\IsotCorr^n_{(\infty, 2)}$ its right adjoint is given by the morphism $Y\leftarrow L\rightarrow X$ with the counit given by the correspondence of correspondences
\[
\xymatrix{
& L\times_X L \ar[dl] \ar[dr] & \\
X & L \ar[u] \ar[d] & X \\
& X \ar[ul] \ar[ur] &
}
\]

Let us discuss the operation adjoint to symplectic implosion in the above sense. We will focus on the quasi-Hamiltonian case for simplicity. The Lie algebra and elliptic cases are treated similarly.

Recall that the group-valued implosion of a Lagrangian $L\rightarrow \left[\frac{G}{G}\right]$ was defined to be the Lagrangian intersection $L\times_{\left[\frac{G}{G}\right]} \left[\frac{\widetilde{G}}{G}\right]$. Dually, given a Lagrangian $L\rightarrow \left[\frac{H}{H}\right]$ we can consider its intersection
\[L\times_{\left[\frac{H}{H}\right]} \left[\frac{\widetilde{G}}{G}\right]\cong [(L\times_{\left[\frac{H}{H}\right]} \widetilde{G})/G].\]

\begin{defn}
The \textit{dual symplectic implosion} of a quasi-Hamiltonian $H$-space $X$ is a quasi-Hamiltonian $G$-space \[X_{dimpl} = [X/H] \times_{\left[\frac{H}{H}\right]} \widetilde{G}.\]
\end{defn}

\begin{remark}
Dual symplectic implosions of conjugacy classes in Levis have previously appeared in \cite{Boa} where they were interpreted in terms of meromorphic connections on the disk.
\end{remark}

As before, it can be given as a Hamiltonian reduction with respect to the universal dual implosion $(H\times H)_{dimpl} = G\times^N 
B$ which is a quasi-Hamiltonian $(G\times H)$-space.

Let us compare quasi-Hamiltonian reductions of the dual implosion and the original space. Let $\cO\subset G$ be a conjugacy class. Then the quasi-Hamiltonian reduction of $X_{dimpl}$ with respect to $G$ along the conjugacy class $\cO$ is
\[X_{dimpl}//_\cO G\cong [X/H]\times_{\left[\frac{H}{H}\right]} \left[\frac{B}{B}\right] \times_{\left[\frac{G}{G}\right]} [\cO/G].\]

Let us compute the fiber product on the right in the two opposite cases.

\begin{enumerate}
\item Suppose $\cO$ is the unit conjugacy class. Then
\[\left[\frac{\widetilde{G}}{G}\right]\times_{\left[\frac{G}{G}\right]} [\pt/G]\cong [(G/B)/G]\cong \B B.\]

Here we regard $\B B$ as a Lagrangian in $\left[\frac{H}{H}\right]$ with the morphism being the composite
\[\B B\rightarrow \B H\rightarrow \left[\frac{H}{H}\right]\]
with the latter map being the inclusion of the unit. The underlying quasi-Hamiltonian $H$-space of $\B B$ is $\B N$.

\item Suppose $\cO$ is the $G$-conjugacy class of a regular semisimple element $\lambda\in H\subset G$. Then its image $\widetilde{\cO}\subset \widetilde{G}$ in the Grothendieck--Springer resolution is a $|W|:1$ cover of $\cO$ where $W$ is the Weyl group. Therefore,
\[[\widetilde{\cO}/G]\rightarrow [\cO/G]\]
is also a $|W|:1$ cover. But $[\cO/G]\cong \B H$, so the quasi-Hamiltonian $H$-space corresponding to $[\widetilde{\cO}/G]$ is identified with the finite set $W$ with the moment map given by sending the whole set to $\lambda\in H$. As we assume $G$ (and hence $H$) is connected, the action of $H$ on $W$ is necessarily trivial.
\end{enumerate}

\begin{thm}
Let $X$ be a quasi-Hamiltonian $H$-space.

The $G$-quasi-Hamiltonian reduction of $X_{dimpl}$ along the unit coincides with the $H$-quasi-Hamiltonian reduction of $X$ along the quasi-Hamiltonian $H$-space $\B N$.

The $G$-quasi-Hamiltonian reduction of $X_{dimpl}$ along the conjugacy class of a regular semisimple element $\lambda\in H$ is a $|W|:1$ cover of the $H$-quasi-Hamiltonian reduction of $X$ along $\lambda$.
\end{thm}

Finally, we can use dual symplectic implosion to show that the procedure of symplectic implosion is not invertible.

\begin{prop}
\label{prop:implosioninvert}
The Lagrangian correspondence
\[
\xymatrix{
&\left[\frac{B}{B}\right] \ar[dl] \ar[dr]& \\
\left[\frac{G}{G}\right] && \left[\frac{H}{H}\right]
}
\]
defining group-valued symplectic implosion is not invertible if the semisimple rank of $G$ is nonzero.
\end{prop}
\begin{proof}
Suppose that a $1$-morphism $L$ from $X$ to $Y$ in $\LagrCorr^n_{(\infty, 1)}$ has an inverse. Then it has an inverse in $\IsotCorr^n_{(\infty, 1)}$ and hence in $\IsotCorr^n_{(\infty, 2)}$. But the latter $(\infty, 2)$-category has adjoints and so the unit and counit of the adjunction for $L$ have to be equivalences.

The right adjoint to $\left[\frac{G}{G}\right]\leftarrow \left[\frac{B}{B}\right]\rightarrow \left[\frac{H}{H}\right]$ is $\left[\frac{H}{H}\right]\leftarrow \left[\frac{B}{B}\right]\rightarrow \left[\frac{G}{G}\right]$ with the counit given by the iterated correspondence
\[
\xymatrix{
& \left[\frac{B}{B}\right]\times_{\left[\frac{G}{G}\right]} \left[\frac{B}{B}\right] \ar[dl] \ar[dr] & \\
\left[\frac{H}{H}\right] & \left[\frac{B}{B}\right] \ar[u] \ar[d] & \left[\frac{H}{H}\right] \\
& \left[\frac{H}{H}\right] \ar[ul] \ar[ur]
}
\]

However, $\left[\frac{B}{B}\right]\rightarrow \left[\frac{H}{H}\right]$ is not an equivalence: the morphism on tangent complexes at the unit element is given by $\b[1]\oplus \b\rightarrow \h[1]\oplus \h$ which has a kernel.
\end{proof}

\end{document}